\documentclass{article}

\usepackage{amsthm}
\usepackage[english]{babel}
\usepackage{amsmath}
\usepackage{amssymb}
\usepackage[%
]{hyperref}
\usepackage{booktabs}
\usepackage{enumerate}
\usepackage[left = 35mm, top = 40mm, bottom = 35mm, right = 25mm]{geometry}
\usepackage{multicol}
\usepackage{bbm}
\usepackage[]{todonotes}
\usepackage{setspace}
\usepackage{mathtools}

\theoremstyle{definition}
\newtheorem{definition}{Definition}
\newtheorem{Remark}[definition]{Remark}
\newtheorem{Example}[definition]{Example}

\theoremstyle{plain}
\newtheorem{theorem}[definition]{Theorem} 
\newtheorem{lemma}[definition]{Lemma}
\newtheorem{Proposition}[definition]{Proposition}
\newtheorem{Corollary}[definition]{Corollary}

\newtheorem*{Question*}{Question}
\newtheorem*{MainProblem}{Main Problem}
\newtheorem*{theorem*}{Theorem}

\newcommand{\soc}{\operatorname{soc}}

\newcommand{\Ann}{\operatorname{Ann}}
\newcommand{\N}{\mathbb{N}}

\newcommand{\Ker}{\operatorname{Ker}}
\newcommand{\lAnn}{\operatorname{lAnn}}
\newcommand{\rAnn}{\operatorname{rAnn}}
\newcommand{\Hom}{\operatorname{Hom}}
\newcommand{\Mat}{\operatorname{Mat}}
\newcommand{\id}{\ensuremath{\mathbbm{1}}}

\renewcommand{\Im}{\operatorname{Im}}

\numberwithin{definition}{section}
\numberwithin{equation}{section}

\title{Ideals in the center of symmetric algebras\\[0.2cm]}
\author{\textbf{Sofia Brenner} \\
\normalsize\emph{Institute for Mathematics, Friedrich Schiller University Jena, Germany}\\
\texttt{sofia.bettina.brenner@uni-jena.de}\\[0.4cm]
\textbf{Burkhard K\"ulshammer} \\
\normalsize\emph{Institute for Mathematics, Friedrich Schiller University Jena, Germany}\\
\texttt{kuelshammer@uni-jena.de}}
\date{\vspace{-0.5cm}}

\begin{document}
\maketitle
\begin{abstract}
\noindent We study symmetric algebras $A$ over an algebraically closed field $F$ in which the Jacobson radical of the center of $A$, the socle of the center of $A$ or the Reynolds ideal of $A$ are ideals.
\end{abstract}
\bigskip

Let $F$ be an algebraically closed field and let $A$ be a finite-dimensional $F$-algebra. As customary, we denote its \emph{center} by $Z(A)$, its \emph{Jacobson radical} by $J(A)$ and its (left) \emph{socle}, the sum of all simple left ideals of $A$, by $\soc(A)$. In this paper, we are interested in the Jacobson radical $J(Z(A))$ and the socle $\soc(Z(A))$ of $Z(A)$ as well as the \emph{Reynolds ideal} $R(A) = \soc(A) \cap Z(A)$ of $A$. All three subspaces are ideals in $Z(A)$. We study the following properties: 

\begin{MainProblem}
	For which finite-dimensional $F$-algebras $A$ is 
	\begin{enumerate}[(P1)]
		\item the Jacobson radical $J(Z(A))$ of $Z(A)$, or
		\item the socle $\soc(Z(A))$ of $Z(A)$, or
		\item the Reynolds ideal $R(A)$ of $A$
	\end{enumerate}
		an ideal in $A$?
\end{MainProblem} 
In this paper, an ideal $I$ of $A$ is always meant to be a two-sided ideal of $A$ and we denote it by $I \trianglelefteq A$. Note that the property (P1) is equivalent to $A \cdot J(Z(A)) \subseteq J(Z(A))$, and the corresponding statement holds for (P2) and (P3). The properties (P1) -- (P3) are trivially satisfied whenever $A$ is commutative. Thus we can view these conditions as weak commutativity properties. 
\bigskip

The question (P1) has already been answered for group algebras and their $p$-blocks by Clarke \cite{CLA69}, Koshitani \cite{KOS78} and K\"ulshammer \cite{KUL20}. The latter paper additionally contains some results on arbitrary symmetric algebras. Moreover, Landrock \cite{LAN20} has proven that $J(Z(A))$ is an ideal of $A$ if $A$ is a symmetric local algebra of dimension at most ten. 
\bigskip

Here we prove that $J(Z(A))$ is an ideal in any symmetric
local algebra $A$ of dimension at most 11, and we present
an example of a symmetric local algebra $A$ of dimension 12
in which $J(Z(A))$ is not an ideal (see Theorem \ref{theo:dim11jzaideal}). In the same spirit, we prove that $\soc(Z(A))$ is an ideal in any symmetric local algebra $A$ of dimension at most 16, and we give an example of a symmetric local algebra $A$ of dimension 20 in which $\soc(Z(A))$ is not an ideal (see Theorem \ref{theo:dim16}). The dimensions 17, 18 and 19 remain open. We also show that, for any symmetric algebra $A$ such that $J(Z(A))$ is an ideal of $A$, and any ideal $I$ of $A$ such that $A/I$ is symmetric, also $J(Z(A/I))$ is an ideal of $A/I$ (see Proposition~\ref{prop:quotientalgebra}) and we prove an analogous result where the radical of the center is replaced by the socle of the center. Suppose that $A_1$ and $A_2$ are nonzero algebras. We prove that $\soc(Z(A_1 \otimes A_2))$ is an ideal in $A_1 \otimes A_2$ if and only if $\soc(Z(A_i))$ is an ideal in $A_i$ for $i=1,2$ (see Proposition \ref{prop:jacobsontensorproduct}). These results will be used in a sequel to this paper dealing with group algebras. 

\section{Finite-dimensional algebras}\label{sec:generalalgebras}
We first investigate the properties (P1) -- (P3) for arbitrary finite-dimensional algebras. We discuss some preliminary results as well as the relation between the three conditions. In the second part, we study the tensor product of two algebras. 

\subsection{Preliminaries}
Throughout this paper, we assume that $F$ is an algebraically closed field. However, many of our results hold for arbitrary fields. All occurring $F$-algebras are supposed to be finite-dimensional, associative and unitary. We write $F\{a_1, \ldots, a_n\}$ for the subspace of an $F$-algebra $A$ spanned by the elements $a_1, \ldots, a_n \in A$. For $a,b \in A$, we set $[a,b] \coloneqq ab -ba$. For subspaces $A_1, A_2$ of $A$, we set $[A_1, A_2] = F\{[a_1, a_2] \colon a_1 \in A_1,\ a_2 \in A_2\}$. The subspace $K(A) = [A,A]$ is called the \emph{commutator space} of $A$.

\begin{lemma}[{\cite[Equation (3)]{KUL91} and \cite[Remark 2.2]{KUL20}}]\label{lemma:commutatorsmallestideal}
Let $A$ be an $F$-algebra.
\begin{enumerate}[(i)]
\item We have $A \cdot K(A) = K(A) \cdot A$, and this is the smallest ideal $I$ of $A$ such that $A/I$ is commutative.
\item For any ideal $I$ of $A$, we have $K(A/I) = K(A) + I/I.$ 
\end{enumerate}
\end{lemma}

In our investigation, we mainly use the following criterion, which is stated in \cite[Lemma 2.1]{KUL20}:

\begin{lemma}\label{lemma:condsocleprod}
Let $A$ be an $F$-algebra.
	\begin{enumerate}[(i)]
		\item For an element $z \in Z(A),$ we have $Az \subseteq Z(A)$ if and only if $K(A) \cdot z = 0$ holds. 
		\item $J(Z(A))$ is an ideal of $A$ if and only if $J(Z(A)) \cdot K(A) = 0$ holds.
		\item $\soc(Z(A))$ is an ideal of $A$ if and only if $\soc(Z(A)) \cdot K(A) = 0$ holds.
		\item $R(A)$ is an ideal of $A$ if and only if $R(A) \cdot K(A) = 0$ holds.
	\end{enumerate}
\end{lemma}

\begin{proof}
The statements (ii) -- (iv) are subject of \cite[Lemma 2.1]{KUL20}, so it remains to prove (i). First assume $Az \subseteq Z(A)$. For $a,b \in A$, we then have $z[a,b] = [za,b] = 0$ since $za \in Z(A)$ holds, so $z$ annihilates $K(A)$. Conversely, if $K(A) \cdot z = 0$ holds, we have $[za,b] = z[a,b] = 0$ for all $a,b \in A$ and hence $za \in Z(A)$ follows.
\end{proof}


We now study the relations between the properties (P1) -- (P3). Recall that the Jacobson radical of $Z(A)$ is given by $J(Z(A)) = J(A) \cap Z(A)$. For a subset $S \subseteq A$, we set $\lAnn_A(S)$ and $\rAnn_A(S)$ to be the left and the right annihilator of $S$ in $A$, respectively, and write $\Ann_A(S)$ if both sets coincide. By \cite[Corollary IV.6.14]{SKO11}, we have $\soc(A) = \rAnn_A(J(A))$ and $\soc(Z(A)) = \Ann_{Z(A)}(J(Z(A)))$. 

\begin{lemma}\label{lemma:raidealnecessary}
Let $A$ be an $F$-algebra. If $\soc(Z(A))$ is an ideal in $A$, then $R(A)$ is an ideal in $A$.
\end{lemma}

\begin{proof}
	Let $\soc(Z(A))$ be an ideal of $A$. Lemma \ref{lemma:condsocleprod} then yields $\soc(Z(A)) \cdot K(A) = 0$. Clearly, $R(A)$ is contained in $\soc(Z(A))$, so we have $R(A) \cdot K(A) \subseteq \soc(Z(A)) \cdot K(A) = 0$ and hence $R(A)$ is an ideal of $A$ by Lemma \ref{lemma:condsocleprod}.
\end{proof}

The converse of Lemma \ref{lemma:raidealnecessary} does not hold. In fact, we will see below (Lemma \ref{lemma:Reynoldsbasic}) that $R(A)$ is an ideal in every symmetric local algebra $A$, but Example~\ref{ex:soc20} presents a symmetric local algebra $A$ in which $\soc(Z(A))$ is not an ideal. Also, there is
no immediate relation between the properties (P1) and (P2) or
(P3).

\begin{Example}\label{ex:firstexample}
	$\null$
	\begin{enumerate}[(i)]
		\item Let $F$ be an algebraically closed field of characteristic $\operatorname{char}(F) = 3$. We consider the free algebra $F \langle X_1, X_2, X_3\rangle$ in variables $X_1, X_2, X_3$ and its quotient algebra 
		$$A \coloneqq F\langle X_1, X_2, X_3 \rangle/ (X_i^3, X_i X_j + X_j X_i)_{i,j =1,2,3,\ i \neq j}.$$ An $F$-basis of $A$ is given by the set
		$\left\{x_1^{r_1} x_2^{r_2} x_3^{r_3} \colon r_1, r_2, r_3 \in \{0,1,2\}\right\},$ where $x_i$ denotes the image of $X_i$ in $A$ for $i =1,2,3$. It is easily verified that $x_1^2$ is contained in $J(Z(A))$, whereas $x_1^2 x_2$ is not. Hence $J(Z(A))$ is not an ideal of $A$. A short direct computation shows
		$$\soc(Z(A)) = F\bigl\{x_1 x_2^2 x_3^2, x_1^2 x_2 x_3^2, x_1^2 x_2^2 x_3, x_1^2 x_2^2 x_3^2\bigr\} \trianglelefteq A.$$ By Lemma \ref{lemma:raidealnecessary}, we also have $R(A) \trianglelefteq A$. 
		\item Now let $F$ be an arbitrary algebraically closed field and consider a non-commutative semisimple $F$-algebra $A$ (e.g.\ the matrix algebra $\Mat_{n}(F)$ for $n \geq 2$). We have $J(Z(A)) = 0 \trianglelefteq A$ and $\soc(Z(A)) = R(A) = Z(A) \not \trianglelefteq A$. 
	\end{enumerate}
\end{Example}

However, in the important special case that $A$ is a local $F$-algebra, we now show that condition (P1) implies (P2). This is mainly due to the following observation:

\begin{lemma}\label{lemma:socinj}
Let $A$ be a local $F$-algebra of dimension at least two. Then $\soc(Z(A)) \subseteq J(Z(A))$ holds.
\end{lemma}
\begin{proof}	
Since $\dim A \geq 2$ holds, we have $J(A) \neq 0$. Let $\ell \in \N$ be the least positive integer with $J(A)^\ell = 0$. Since $A$ is of the form $A = F \cdot 1 \oplus J(A)$, we have $J(A)^{\ell -1} \subseteq Z(A)$. This implies $J(Z(A)) = J(A) \cap Z(A) \neq 0$. It follows that $\soc(Z(A))$ is a proper (nilpotent) ideal of $Z(A)$, which implies $\soc(Z(A)) \subseteq J(Z(A))$. 
\end{proof}

We therefore obtain the following implication:

\begin{Corollary}
Let $A$ be a local $F$-algebra. If $J(Z(A))$ is an ideal of $A$, then $\soc(Z(A))$ is an ideal of~$A$. 
\end{Corollary} 

\begin{proof}
	If $A \cong F$ holds, then $A$ is commutative and hence $\soc(Z(A)) = \soc(A)$ is an ideal in $A$. Now assume $\dim A \geq 2$.
	By Lemmas \ref{lemma:socinj} and \ref{lemma:condsocleprod}, we have $\soc(Z(A)) \cdot K(A) \subseteq J(Z(A)) \cdot K(A) = 0.$ Hence $\soc(Z(A))$ is an ideal of $A$ again by Lemma \ref{lemma:condsocleprod}. 
\end{proof}

\subsection{Tensor products}\label{sec:tensorproduct}
We now study the properties (P1) -- (P3) for the tensor product $A_1 \otimes A_2 \coloneqq A_1 \otimes_F A_2$ of two nonzero $F$-algebras $A_1$ and $A_2$. The following identities are well-known:
\begin{itemize}
\item $Z(A_1 \otimes A_2) = Z(A_1) \otimes Z(A_2)$
\item $J(A_1 \otimes A_2) = J(A_1) \otimes A_2 + A_1 \otimes J(A_2)$
\item $K(A_1 \otimes A_2) = K(A_1) \otimes A_2 + A_1 \otimes K(A_2).$
\end{itemize}

We now prove corresponding formulas for the socle and the Reynolds ideal of $A_1 \otimes A_2$, which are probably known as well:
\begin{lemma}\label{lemma:soctensor}
Let $A_1$ and $A_2$ be $F$-algebras.
\begin{enumerate}[(i)]
\item We have $\soc(A_1 \otimes A_2) = \soc(A_1) \otimes \soc(A_2).$ 
\item We have $R(A_1 \otimes A_2) = R(A_1) \otimes R(A_2)$. 
\end{enumerate}
\end{lemma}

\begin{proof}
Let $n_1, n_2 \in \N_0$ denote the dimensions of $A_1$ and $A_2$, respectively.
\begin{enumerate}[(i)]
\item We first show $\soc(A_1) \otimes A_2 = \rAnn_{A_1 \otimes A_2}(J(A_1) \otimes A_2)$. For $s \in \soc(A_1)$, $a,b \in A_2$ and $j \in J(A_1),$ we have $ (j \otimes a) \cdot (s \otimes b) = js \otimes ab = 0 \otimes ab =  0,$ so $\soc(A_1) \otimes A_2$ is contained in the right annihilator of $J(A_1) \otimes A_2$. For the converse inclusion, we choose $F$-bases $\{v_1, \ldots, v_{n_1}\}$ of $A_1$ and $\{w_{1}, \ldots, w_{n_2}\}$ of $A_2$. The set $\{v_i \otimes w_k \colon i = 1, \ldots, n_1, \, k = 1, \ldots, n_2\}$ is an $F$-basis of $A_1 \otimes A_2.$ Consider an element $x \in \rAnn(J(A_1) \otimes A_2)$ and write $x \coloneqq \sum_{i = 1}^{n_1} \sum_{k = 1}^{n_2} \lambda_{ik} (v_i \otimes w_k)$ with $\lambda_{ik} \in F$ ($i = 1, \ldots, n_1$, $k = 1, \ldots, n_2$). For any $j \in J(A_1),$ we obtain
	$$0 = (j \otimes 1)\cdot x = \sum_{i = 1}^{n_1} \sum_{k = 1}^{n_2} \lambda_{ik} (j v_i \otimes w_k) = \sum_{k = 1}^{n_2} \left(\sum_{i = 1}^{n_1} \lambda_{ik} j  v_i\right) \otimes w_k.$$
	For $k = 1, \ldots, n_2$, this yields $0 = \sum_{i = 1}^{n_1} \lambda_{ik} j v_i  = j \cdot \sum_{i=1}^{n_1} \lambda_{ik} v_i $. We obtain $\sum_{i = 1}^{n_1} \lambda_{ik} v_i \in \soc(A_1)$, which implies $x \in \soc(A_1) \otimes A_2$. Analogously, we show $A_1 \otimes \soc(A_2) = \rAnn_{A_1 \otimes A_2} (A_1 \otimes J(A_2))$. This yields
	\begin{alignat*}{1}
	\soc(A_1 \otimes A_2) &= \rAnn_{A_1 \otimes A_2} (J(A_1) \otimes A_2 + A_1 \otimes J(A_2))  \\
	&= \rAnn_{A_1 \otimes A_2} (J(A_1) \otimes A_2) \cap \rAnn_{A_1 \otimes A_2}(A_1 \otimes J(A_2)) \\
	&= (\soc(A_1) \otimes A_2) \cap (A_1 \otimes \soc(A_2)) \\
	&= \soc(A_1) \otimes \soc(A_2). 
	\end{alignat*}
\item With (i), we obtain
\begin{alignat*}{1}
R(A_1 \otimes A_2) &= \soc(A_1 \otimes A_2) \cap Z(A_1 \otimes A_2) \\
&= \left(\soc(A_1) \otimes \soc(A_2)\right) \cap \left(Z(A_1) \otimes Z(A_2)\right) \\
&= \left(\soc(A_1) \cap Z(A_1)\right) \otimes \left(\soc(A_2) \cap Z(A_2)\right) \\
&= R(A_1) \otimes R(A_2). \tag*{\qedhere}
\end{alignat*}
\end{enumerate}
\end{proof}

\begin{lemma}\label{lemma:idealtensor}
Let $A_1$ and $A_2$ be $F$-algebras and let $U_1$ and $U_2$ be nonzero subspaces of $A_1$ and $A_2$, respectively. Then $U_1 \otimes U_2$ is an ideal in $A_1 \otimes A_2$ if and only if $U_i$ is an ideal in $A_i$ for $i =1,2$.
\end{lemma}

\begin{proof}
Assume that $U_i$ is an ideal of $A_i$ for $i = 1,2$. For $a_1 \in A_1$ and $a_2 \in A_2$, we then have $$(a_1 \otimes a_2) \cdot (U_1 \otimes U_2) \subseteq U_1 \otimes U_2,$$ which shows that $U_1 \otimes U_2$ is an ideal of $A_1 \otimes A_2$.
\bigskip
		
Now assume conversely that $U_1 \otimes U_2$ is an ideal of $A_1 \otimes A_2$. We choose $F$-bases $\{v_1, \ldots, v_{n_1}\}$ of~$A_1$ and $\{w_1, \ldots, w_{n_2}\}$ of $A_2$ such that $\{v_1, \ldots, v_{k_1}\}$ and $\{w_1, \ldots, w_{k_2}\}$ are bases of $U_1$ and $U_2$ for some $k_i \in \{1, \ldots, n_i\}$ ($i = 1,2$), respectively. The set $\{v_{j_1} \otimes w_{j_2} \colon 1 \leq j_i \leq k_i \text{ for } i = 1,2\}$ is an $F$-basis for $U_1 \otimes U_2$. 
\bigskip
		
We show that $a_1 v_i \in U_1$ holds for all $a_1 \in A_1$ and $i = 1, \ldots, k_1.$ To this end, we set $a \coloneqq a_1 \otimes 1 \in A_1 \otimes A_2$ and $v \coloneqq v_i \otimes w_1 \in U_1 \otimes U_2.$ Since $av \in U_1 \otimes U_2$ holds by assumption, there exist coefficients $\lambda_{rt} \in F$ for $1 \leq r \leq k_1$, $1 \leq t \leq k_2$ with 
$av = \sum_{r,t} \lambda_{rt} v_r \otimes w_t.$
Expressing $a_1 v_i = \sum_{d= 1}^{n_1} \mu_d v_d$ in terms of the basis of $A_1$ (with $\mu_1, \ldots, \mu_{n_1} \in F$) yields 
$$av = (a_1 \otimes 1) \cdot (v_i \otimes w_1) = a_1 v_i \otimes w_1 = \left(\sum_{d=1}^{n_1} \mu_d v_d\right) \otimes w_1 = \sum_{d= 1}^{n_1} \mu_d (v_d \otimes w_1).$$ By comparing the coefficients in the two expressions for $av$, we obtain $\mu_d = 0 $ for $d > k_1$. This shows $a_1 v_i  \in U_1$, which proves that $U_1$ is an ideal of $A_1.$ For $U_2$, we proceed analogously.
\end{proof}

\begin{Proposition}\label{prop:jacobsontensorproduct}
Let $A_1$ and $A_2$ be nonzero $F$-algebras. 
\begin{enumerate}[(i)]
\item $\soc(Z(A_1 \otimes A_2))$ is an ideal of $A_1 \otimes A_2$ if and only if $\soc(Z(A_i))$ is an ideal of $A_i$ for $i = 1, 2$.
\item $R(A_1 \otimes A_2)$ is an ideal of $A_1 \otimes A_2$ if and only if $R(A_i)$ is an ideal of $A_i$ for $i =1,2$. 
\end{enumerate}
\end{Proposition}

\begin{proof}
Since $\soc(Z(A_1 \otimes A_2)) = \soc(Z(A_1) \otimes Z(A_2)) = \soc(Z(A_1)) \otimes \soc(Z(A_2))$ and $R(A_1 \otimes A_2) = R(A_1) \otimes R(A_2)$ hold by Lemma \ref{lemma:soctensor}, the claim follows by Lemma \ref{lemma:idealtensor}.
\end{proof}

In contrast, the corresponding statement for the Jacobson radical of $Z(A_1 \otimes A_2)$ does not hold:

\begin{Example}
Consider the quotient algebra $A = F[X]/(X^2)$ of the polynomial ring $F[X]$ and let $M \coloneqq \Mat_2(A)$ be the algebra of $2\times 2$-matrices with entries in $A$. Note that $M \cong \Mat_2(F) \otimes_F A$ holds. We have $Z(M) = A \cdot \id \cong A$ and hence $J(Z(M)) = J(A) \cdot \id$. It is easily seen that $J(Z(M))$ is not closed under multiplication with arbitrary elements of $M$, so $J(Z(M))$ is not an ideal in $M$. On the other hand, we have $J(Z(A)) \trianglelefteq A$ since $A$ is commutative and $J(Z(\Mat_2(F)))$ is an ideal in $\Mat_2(F)$ since $\Mat_2(F)$ is semisimple (see Example \ref{ex:firstexample}). 

\end{Example}

\section{Symmetric algebras}\label{sec:symmetricalgebrasintro}
Let $F$ be an algebraically closed field. In this section, we investigate the main problem for symmetric algebras. In particular, we study the transition to certain quotient algebras.
\bigskip

A finite-dimensional $F$-algebra $A$ is called \emph{symmetric} if it admits a non-degenerate associative symmetric bilinear form $\beta \colon A \times A \to F$. The kernel of the associated linear form $\lambda \colon A \to F,\ a \mapsto \beta(1,a)$ then contains the commutator space $K(A),$ but no nonzero one-sided ideal of $A$ (see \cite[Theorem IV.2.2]{SKO11}). For a subspace $X$ of $A$, we consider its orthogonal space $X^\perp = \{a \in A \colon \beta(a,x) = 0 \text{ for } x \in X\}$ with respect to $\beta$.
\begin{lemma}[{\cite[Equations (28) -- (32), (35)]{KUL91}}]\label{lemma:propertiesperp}
	Let $A$ be a symmetric $F$-algebra and consider subspaces $X$ and $Y$ of $A.$ Then the following hold:
	\begin{enumerate}[(i)]
		\item $\dim X + \dim X^\perp = \dim A$.
		\item $(X^\perp)^\perp = X$.
		\item $Y \subseteq X$ implies $X^\perp \subseteq Y^\perp.$
		\item We have $(X \cap Y)^\perp = X^\perp + Y^\perp$ and $(X+Y)^\perp = X^\perp \cap Y^\perp.$
		\item For an ideal $I$ of $A$, we have $I^\perp = \lAnn_A(I) = \rAnn_A(I)$, and $I^\perp$ is an ideal of $A$ as well. In particular, we obtain $J(A) = \soc(A)^\perp$.
		\item $K(A)^\perp = Z(A)$.
	\end{enumerate}
\end{lemma}

The symmetric algebras in which the Reynolds ideal is an ideal can be characterized as follows:

\begin{lemma}\label{lemma:Reynoldsbasic}
	Let $A$ be a symmetric $F$-algebra. Then $R(A)$ is an ideal of $A$ if and only if $A$ is basic. In this case, we have $R(A) = \soc(A)$.
\end{lemma}

\begin{proof}
	Assume that $R(A)$ is an ideal of $A$. By \cite[Remark 3.1]{KUL20}, we obtain $R(A) = A \cdot R(A) = \soc(A)$ and $A$ is basic. Conversely, assume that $A$ is a basic, which is equivalent to $A/J(A)$ being commutative by \cite[Proposition II.6.19]{SKO11}. This implies $K(A) \subseteq J(A)$ and hence $\soc(A) =  J(A)^\perp \subseteq K(A)^\perp = Z(A)$ follows by Lemma \ref{lemma:propertiesperp}, so $R(A) = \soc(A) \cap Z(A) = \soc(A)$ is an ideal of~$A.$ 
\end{proof}

In the following, we therefore focus on the study of $J(Z(A))$ and $\soc(Z(A))$.

%

\subsection{Transition to quotient algebras}\label{sec:quotientssymmetric}
In this part, we consider various quotient algebras of $A$.

\begin{lemma}\label{lemma:idealsymmetricalternative}
Let $A$ be a symmetric $F$-algebra.
	\begin{enumerate}[(i)]
		\item $J(Z(A)) \trianglelefteq A$ holds if and only if $K(\bar{A})$ is an ideal of $\bar{A} \coloneqq A/\soc(A).$
		\item $\soc(Z(A)) \trianglelefteq A$ holds if and only if $K(\bar{A})$ is an ideal of $\bar{A} \coloneqq A/A \cdot J(Z(A)).$
	\end{enumerate}
\end{lemma}	

\begin{proof}
	By Lemma \ref{lemma:propertiesperp}, $J(Z(A))$ is an ideal of $A$ if and only if $J(Z(A))^\perp$ is. Moreover, we have
	$$J(Z(A))^\perp = (J(A) \cap Z(A))^\perp = J(A)^\perp + Z(A)^\perp = \soc(A) + K(A).$$ Since $\soc(A)$ is an ideal of $A$, $J(Z(A))^\perp$ is an ideal of $A$ if and only if $A \cdot K(A) \cdot A \subseteq K(A) + \soc(A)$ holds. Setting $\bar{A} \coloneqq A/\soc(A)$, this is equivalent to $K(\bar{A}) = K(A) + \soc(A)/\soc(A)$ being an ideal of $\bar{A}$ (see Lemma \ref{lemma:commutatorsmallestideal}). Similarly, $\soc(Z(A))$ is an ideal of $A$ if and only if $\soc(Z(A))^\perp$ is. Note that 
	$\Ann_A(J(Z(A))) = \Ann_A(A \cdot J(Z(A))) = (A \cdot J(Z(A)))^\perp$ follows by Lemma \ref{lemma:propertiesperp}\,(v), which yields
	$$\soc(Z(A)) = Z(A) \cap \Ann_A(J(Z(A)) = Z(A) \cap (A \cdot J(Z(A)))^\perp.$$ By Lemma \ref{lemma:propertiesperp}, this implies
	$\soc(Z(A))^\perp = Z(A)^\perp + A \cdot J(Z(A)) = K(A) + A \cdot J(Z(A)).$
	Hence $\soc(Z(A))$ is an ideal of $A$ if and only if $K(A) + A \cdot J(Z(A))$ is, which is equivalent to $K(\bar{A})$ being an ideal of $\bar{A} \coloneqq A/A \cdot J(Z(A))$. 
\end{proof}

\begin{Remark}\label{rem:ka}
	By Lemma \ref{lemma:propertiesperp}\,(v), $K(A)$ is an ideal of $A$ if and only if $K(A)^\perp = Z(A)$ is, i.e., if and only if $A$ is commutative. If $J(Z(A))$ or $\soc(Z(A))$ are ideals of $A$, then the corresponding algebras~$\bar{A}$ defined in Lemma \ref{lemma:idealsymmetricalternative} are therefore either commutative or non-symmetric.
\end{Remark}

We now study quotient algebras of $A$ which are again symmetric. Note that this is an additional condition which is not satisfied for arbitrary quotients of $A$:

\begin{lemma}\label{lemma:quotientalgebrasymmetric}
	Let $A$ be a symmetric $F$-algebra and let $I \trianglelefteq A$ be an ideal such that $A/I$ is symmetric with corresponding linear form $\bar{\lambda}$. Then there exists $z \in Z(A)$ with $I = (Az)^\perp$ such that $\bar{\lambda}(a+ I) = \lambda(az)$ holds for all $a \in A.$ Conversely, for $z \in Z(A)$, the algebra $A/(Az)^\perp$ is symmetric with respect to a linear form $\bar{\lambda}$ of the above shape.
\end{lemma}

\begin{proof}
	The proof is given in \cite[pages 429 -- 430]{KUL91}.
\end{proof}

With this characterization, we can simplify the criterion for $J(Z(A)) \trianglelefteq A$ in case that $A$ is local.

\begin{lemma}\label{lemma:aicommutative}
	Let $A$ be a symmetric local algebra. Then $J(Z(A)) \trianglelefteq A$ holds if and only if for all ideals $0 \neq I \trianglelefteq A$ such that $A/I$ is symmetric, it follows that $A/I$ is commutative. 
\end{lemma}

\begin{proof}
	First assume $J(Z(A)) \trianglelefteq A$ and let $I \trianglelefteq A$ be a nonzero ideal such that $A/I$ is symmetric. Since~$A$ is local, the ideal $I^\perp$ is contained in $J(A)$. This yields $\soc(A) \cdot I^\perp = 0$ and hence $\soc(A) \subseteq (I^\perp)^\perp = I$. The algebra $A/I$ is therefore isomorphic to a quotient of $\bar{A} \coloneqq A/\soc(A)$. By Lemma \ref{lemma:idealsymmetricalternative}, we have $K(\bar{A}) \trianglelefteq \bar{A}$ and hence $K(A/I) \trianglelefteq A/I$ follows by Lemma \ref{lemma:commutatorsmallestideal}. Since $A/I$ is symmetric, this implies $K(A/I) = 0$ (see Remark \ref{rem:ka}), so $A/I$ is commutative. Conversely, assume that $K(A/I) = 0$ holds for every ideal $ 0 \neq I \trianglelefteq A$ for which the quotient $A/I$ is symmetric. By Lemma \ref{lemma:quotientalgebrasymmetric}, we have $ I = (Az)^\perp$ for some $z \in J(Z(A))$. Lemma \ref{lemma:commutatorsmallestideal} then yields 
	$$K(A) \subseteq \bigcap_{z \in J(Z(A))} (Az)^\perp = \left(\sum_{z \in J(Z(A))} Az\right)^\perp = (A \cdot J(Z(A)))^\perp$$ and hence $J(Z(A)) \cdot K(A) \subseteq J(Z(A)) \cdot (A \cdot J(Z(A)))^\perp = 0$. By Lemma \ref{lemma:condsocleprod}, $J(Z(A))$ is an ideal of~$A$. 
\end{proof}

\subsection{Symmetric quotient algebras}
Let $A$ be a symmetric algebra. The aim of this section is to prove that the properties $J(Z(A)) \trianglelefteq A$ and $\soc(Z(A)) \trianglelefteq A$ are inherited by symmetric quotient algebras of $A$. 
\bigskip

Let $I$ be an ideal of $A$ such that $\bar{A} \coloneqq A/I$ is symmetric with respect to a bilinear form $\bar{\beta} \colon \bar{A} \times \bar{A} \to F$. By Lemma \ref{lemma:quotientalgebrasymmetric}, we have $I = (Az)^\perp$ for some $z \in Z(A)$ and a symmetrizing linear form $\bar{\lambda} \colon \bar{A} \to F$ is given by $\bar{\lambda}(a+I) = \lambda(az)$ for all $a \in A$. 
Let $\nu \colon A \to \bar{A},\ a \mapsto \bar{a} \coloneqq a + I$ be the canonical projection map. In the following, we consider its adjoint map $\nu^* \colon \bar{A} \to A$ defined by requiring
\begin{equation}\label{eq:nustardef}
\beta(\nu^*(\bar{x}), y) = \bar{\beta}(\bar{x},\nu(y)) \text{ for all } x,y \in A.
\end{equation}

\begin{lemma}\label{lemma:propnustar}
The map $\nu^*$ has the following properties:
\begin{enumerate}[(i)]
\item It is explicitly given by $\nu^*(\bar{x}) = xz$ for all $x \in A$. 
\item For all $x,y \in A$, we have $\nu^*(\bar{x}) \cdot y = \nu^*(\bar{x} \cdot \bar{y})$ and $x \cdot \nu^*(\bar{y})  =\nu^*(\bar{x} \cdot \bar{y})$.
\item The map $\nu^*$ is injective. Thus it induces an
isomorphism of $A$-bimodules between $\bar{A}$
and $Az$.

\end{enumerate}
\end{lemma}

\begin{proof}
$\null$
\begin{enumerate}[(i)]
\item The defining property \eqref{eq:nustardef} of $\nu^*$ is equivalent to
$\lambda\bigl(\nu^*(\bar{x}) \cdot y\bigr)  = \bar{{\lambda}}(\bar{x} \cdot \bar{y}) = \lambda(xyz) = \lambda(xzy)$ for all $x,y \in A$. Hence the right ideal $(\nu^*(\bar{x})-xz)A $ is contained in the kernel of $\lambda$. Since $\Ker(\lambda)$ contains no nontrivial one-sided ideals, this implies $\nu^*(\bar{x}) = xz$. 
\item This directly follows from (i).
\item This follows from the fact that the map $\nu$ is surjective. \qedhere
\end{enumerate}	
\end{proof}

\begin{lemma}\label{lemma:nustarmapscentertocenterandsoctosoc}
We obtain the following relations:
	\begin{enumerate}[(i)]
		\item $\nu^*\bigl(Z(\bar{A})\bigr) = Z(A) \cap \Im(\nu^*) = Z(A) \cap Az$.
		\item $\nu^\ast(J(Z(A))) = Z(A) \cap \nu^{-1}(\soc({\bar A}))^\perp \subseteq J(Z(A)) \cap Az$.
		\item For $x \in Z(\bar{A})$, we have $\nu^\ast(x) \in \soc(Z(A))$ if and only if $x \cdot \nu(J(Z(A))) = 0$ holds. In particular, we obtain $\nu^*\bigl(\soc(Z(\bar{A}))\bigr) \subseteq \soc(Z(A))$.
	\end{enumerate}	
\end{lemma}

\begin{proof}
Let $U$ be any subspace of $\bar{A}$, and let $x \in A$.
Then $x \in \nu^\ast(U)^\perp$ is equivalent to
$0 = \beta(\nu^\ast(U),x) = {\bar\beta}(U,\bar{x})$,
which implies $\bar{x} \in U^\perp$. This shows
$\nu^\ast(U)^\perp = \nu^{-1}(U^\perp)$ and
$\nu^\ast(U) = \nu^{-1}(U^\perp)^\perp$.
	\begin{enumerate}[(i)]
		\item For $U = Z(\bar{A})$, we have $U^\perp = K(\bar{A}) = K(A)+I/I$ and $\nu^{-1}(U^\perp) = K(A) + I$. Thus $\nu^\ast(U) =(K(A)+I)^\perp = Z(A) \cap I^\perp = Z(A) \cap Az$ follows.
		\item Let $U = J(Z(\bar{A})) = Z(\bar{A}) \cap J(\bar{A})$.
		Then we have $U^\perp = K(\bar{A}) + \soc(\bar{A})$ and
		$\nu^{-1}(U^\perp) = K(A) + \nu^{-1}(\soc(\bar{A}))$, so that
		$$\nu^\ast(U) = (K(A) + \nu^{-1}(\soc(\bar{A})))^\perp =
		Z(A) \cap \nu^{-1}(\soc({\bar A}))^\perp.$$
		Moreover, since $\nu(\soc(A) + I) = \nu(\soc(A)) \subseteq \soc(\bar{A})$ holds,
		we have $\soc(A) + I \subseteq \nu^{-1}(\soc(\bar{A}))$. This implies $\nu^{-1}(\soc(\bar{A}))^\perp \subseteq (\soc(A) + I)^\perp = J(A) \cap I^\perp$ and $$Z(A) \cap \nu^{-1}(\soc(\bar{A}))^\perp
		\subseteq Z(A) \cap J(A) \cap I^\perp = J(Z(A)) \cap Az.$$ 
		\item For $x \in Z(\bar{A})$, we have $\nu^\ast(x) \in Z(A)$. By using Lemma \ref{lemma:propnustar} and the injectivity of $\nu^\ast$, we obtain the equivalence
		$$\nu^\ast(x) \in \soc(Z(A)) \Leftrightarrow \nu^\ast(x) \cdot J(Z(A)) = 0 \Leftrightarrow \nu^\ast\bigl(x \cdot \nu(J(Z(A)))\bigr) = 0 \Leftrightarrow x \cdot \nu\bigl(J(Z(A))\bigr) = 0.$$
		Now let $x \in \soc(Z(\bar{A}))$. We then have $\nu\bigl(J(Z(A))\bigr) \subseteq J(\bar{A}) \cap Z(\bar{A}) = J(Z(\bar{A}))$, which implies $x \cdot \nu\bigl(J(Z(A))\bigr) = 0$. The above equivalence then yields $\nu^\ast(x) \in \soc(Z(A))$. \qedhere
	\end{enumerate}
\end{proof}

\begin{Remark}
Lemma \ref{lemma:nustarmapscentertocenterandsoctosoc}\,(iii) shows that $\soc(Z(A)) \cap \Im \nu^\ast$ is precisely the image under $\nu^\ast$ of the annihilator of $\nu(J(Z(A)))$ in $Z(\bar{A})$.
\end{Remark}

We now prove that the properties $J(Z(A)) \trianglelefteq A$ and $\soc(Z(A)) \trianglelefteq A$ are inherited by symmetric quotient algebras of $A$: 

\begin{Proposition}\label{prop:quotientalgebra}
	Let $A$ be a symmetric $F$-algebra and consider an ideal $I \trianglelefteq A$ for which $\bar{A} \coloneqq A/I$ is symmetric. 
	\begin{enumerate}[(i)]
		\item If $J(Z(A)) \trianglelefteq A$ holds, then $J(Z(\bar{A}))$ is an ideal of $\bar{A}$.
		\item If $\soc(Z(A)) \trianglelefteq A$ holds, then $\Ann_{Z(\bar{A})}(\nu(J(Z(A))))$
		and $\soc(Z(\bar{A}))$ are ideals of $\bar{A}$.
	\end{enumerate}
\end{Proposition}

\begin{proof}
$\null$
\begin{enumerate}[(i)]
	\item Suppose that $J(Z(A))$ is an ideal in $A$. By Lemma \ref{lemma:condsocleprod}, we have $J(Z(A)) \cdot K(A) = 0$. In particular, this yields $0 =
	\nu^\ast(J(Z(\bar{A}))) \cdot K(A) =
	\nu^\ast(J(Z(\bar{A})) \cdot K(\bar{A}))$ (see Lemmas \ref{lemma:nustarmapscentertocenterandsoctosoc} and \ref{lemma:propnustar}). Since $\nu^\ast$ is injective, this implies
	$J(Z(\bar{A})) \cdot K(\bar{A}) = 0$. Hence $J(Z(\bar{A}))
	\trianglelefteq \bar{A}$ follows again by Lemma \ref{lemma:condsocleprod}. 
	
	\item Suppose that $\soc(Z(A))$ is an ideal in $A$. Let $a \in A$ with $\bar{a} \in \Ann_{Z(\bar{A})}(\nu(J(Z(A))))$. We then have $\bar{a} \cdot \nu(J(Z(A))) = 0$, so $\nu^\ast({\bar a}) \in \soc(Z(A))$ follows by Lemma \ref{lemma:nustarmapscentertocenterandsoctosoc}\,(iii). Since $\soc(Z(A))$ is an ideal of $A$, this implies
	$$0 = \nu^\ast(\bar{a}) \cdot K(A) = \nu^\ast\bigl(\bar{a} \cdot K(\bar{A})\bigr).$$ Thus $\bar{a} \cdot K(\bar{A}) = 0$ follows. By Lemma \ref{lemma:condsocleprod}\,(i), this yields $\bar{A} \bar{a} \subseteq Z(\bar{A})$ and we conclude that $\bar{A} \bar{a} \subseteq \Ann_{Z(\bar{A})}(\nu(J(Z(A))))$ holds.	The statement concerning $\soc(Z(\bar{A}))$ can be proven similarly to~(i). \qedhere
\end{enumerate}
\end{proof}

The following example demonstrates that in general, the properties (P1) and (P2) are not transferred to quotient algebras, that is, the prerequisites of Proposition \ref{prop:quotientalgebra} are necessary.

\begin{Example}\label{ex:counterexamplesymmetricnecessary}
	Let $F$ be an algebraically closed field of characteristic $\operatorname{char}(F) = 5$ and let $q \in F^\times$ be an element of order $24$. 
	We consider the (non-symmetric) algebra 
	$$A \coloneqq F\langle X_1, X_2, X_3\rangle/( X_1^5, X_2^5, X_3^2, X_1 X_2 + X_2 X_1, X_3 X_1 - q X_1 X_3, X_3 X_2 - q X_2 X_3).$$
	Here, $F\langle X_1, X_2, X_3 \rangle$ denotes the free $F$-algebra in variables $X_1, X_2, X_3$. An $F$-basis of $A$ is given by $$\left\{ x_1^{\ell_1} x_2^{\ell_2}x_3^{\ell_3} \colon \ell_1, \ell_2 \in \{0, \ldots, 4\},\ \ell_3 \in \{0,1\}\right\},$$ where $x_i$ denotes the image of $X_i$ in $A$ for $i =1,2,3$. One can verify directly that $Z(A) = F\{1, x_1^4 x_2^4 x_3\}$ is two-dimensional and that $J(Z(A)) = \soc(Z(A)) = \soc(A)$ is an ideal of $A$.
	Now consider the algebra $$B \coloneqq F\langle X_1, X_2\rangle /(X_1^2, X_2^4, X_1 X_2 + X_2 X_1),$$ which can be viewed as a quotient algebra of $A$.
We write $y_1$ and $y_2$ for the images of $X_1$ and $X_2$ in $B$, respectively. An $F$-basis of $B$ is given by $$\left\{y_1^{\ell_1} y_2^{\ell_2} \colon \ell_1 \in \{0,1\},\ \ell_2 \in \{0, \ldots, 3\}\right\}.$$
	A short computation shows that $J(Z(B)) = \soc(Z(B)) = F\{y_2^2,y_1 y_2^3\}$ is two-dimensional. However, $$B \cdot J(Z(B)) = B \cdot \soc(Z(B)) = F\left\{y_2^2, y_2^3, y_1 y_2^2, y_1 y_2^3\right\}$$ is of dimension four, so $J(Z(B)) = \soc(Z(B))$ is not an ideal in $B$.
\end{Example}

\section{Trivial extension algebras}\label{sec:trivialextension}
Let $F$ be an algebraically closed field. In the following, we consider trivial extension algebras, which arise in various contexts in the representation theory of finite-dimensional algebras.
\bigskip

For an $F$-vector space $V$, we set $V^\ast \coloneqq \Hom_F(V,F)$ to be the space of $F$-linear forms on $V$. Let $A$ be an $F$-algebra. Recall that $A^\ast$ becomes an $A$-$A$-bimodule by setting $(af)(x) \coloneqq f(xa) \text{ and } (fa)(x) \coloneqq f(ax)$
for $x,a \in A$ and $f \in A^\ast.$ The trivial extension algebra $T \coloneqq T(A)$ of $A$ is the vector space $A \oplus A^\ast$, endowed with the multiplication law
$$(a,f) \cdot (b,g) \coloneqq (ab, ag + fb) \text{ for all }a,b \in A \text{ and }f,g \in A^\ast.$$
We denote the elements of $T$ by tuples $(a,f)$ with $a \in A$ and $f \in A^\ast$.
By \cite[Proposition 3.1]{BES07}, $T$ is a symmetric algebra with symmetrizing linear form $\lambda \colon T \to F,\ (a,f) \mapsto f(1)$. 
\bigskip

In the following, we identify subspaces $V \subseteq A$ and $W \subseteq A^\ast$ with the subspaces $V \oplus 0$ and $0 \oplus W$ of $T$, respectively. In this way, $A$ can be viewed as a subalgebra of $T$. Similarly, we identify $A^\ast$ with the ideal $0 \oplus A^\ast$ of $T$, which squares to zero. For a subspace $U$ of $A$, we view $(A/U)^\ast$ as a subset of $A^\ast$ by identifying the map $f \colon A/U \to F$ with $f^\wedge \colon A \to F$ defined by setting $f^{\wedge}(x) = f(x+U)$ for all $x \in A$.
\bigskip

We first determine the substructures of $T$ investigated in this paper, thereby extending some results of \cite{BES07} and \cite{CHL91}. 

\begin{lemma}\label{lemma:subspacest}
Let $A$ be an $F$-algebra and let $T \coloneqq T(A)$ be the trivial extension algebra of $A$. Then the following identities hold: 
\begin{enumerate}[(i)]
\item $Z(T) = Z(A) \oplus (A/K(A))^\ast$
\item $K(T) = K(A) \oplus \left[A,A^\ast\right]$, where $[A,A^\ast]$ is the $F$-subspace of $A^\ast$ spanned by the elements $af - fa$ with $a \in A$ and $f \in A^\ast$.
\item $J(T) = J(A) \oplus A^\ast$
\item $J(Z(T)) = J(Z(A)) \oplus (A/K(A))^\ast$
\item $\soc(T) = 0 \oplus (A/J(A))^\ast$
\item $\soc(Z(T)) = \{b \in \soc(Z(A)) \colon Ab \subseteq K(A)\} \oplus  (A/K(A)+ A \cdot J(Z(A))^\ast$
\item $R(T) = 0 \oplus (A/K(A) + J(A))^\ast$.
\end{enumerate}
\end{lemma}

\begin{proof}
The statements of (i) and (ii) are proven in \cite{BES07} as well as \cite{CHL91}. Since $J(A) \oplus A^\ast$ is a nilpotent ideal of~$T$, we have $J(A) \oplus A^\ast \subseteq J(T)$. On the other hand, we have $J(T) = (J(T) \cap A) \oplus A^\ast$ and $J(T) \cap A$ is a nilpotent ideal of $A$, which yields $J(T) \cap A \subseteq J(A)$. This shows the identity in (iii). Combined with (i), this yields the formula for $J(Z(T))$ given in (iv). An easy calculation shows that $0 \oplus (A/J(A))^\ast$ is contained in $\soc(T)$. Since $\dim 0 \oplus (A/J(A))^\ast = \dim T - \dim J(T)$ follows by (iii), we obtain the equality given in (v).
%
\bigskip

Now we show (vi). To this end, set $I \coloneqq K(A) + A \cdot J(Z(A))$. 
Consider $t \in \soc(Z(T))$, $a \in J(Z(A))$ and $f \in (A/K(A))^\ast$. By (i), we have $t = (b,g)$ with $b \in Z(A)$ and $g \in (A/K(A))^\ast$. Moreover, (iv) implies $(a,f) \in J(Z(T))$. Thus $0 = (b,g)(a,f) = (ba,bf+ga)$, i.e., we have $ba = 0$ and $0 = (bf+ga)(x) = f(xb) + g(ax)$ for $x \in A$. Since $a$ is arbitrary, this forces $b \in \soc(Z(A))$ and $f(xb) = 0$ for $x \in A$, i.e.\ $f(Ab) = 0$. Since $f$ is arbitrary, we obtain $Ab \subseteq K(A)$ and $g(ax) = 0$ for $x \in A$, so that $g(aA) = 0$. Since $a$ is arbitrary, this also implies $g \in (A/I)^\ast$. 
\bigskip

Now let $b \in \soc(Z(A))$ with $Ab \subseteq K(A)$ and $g \in (A/I)^\ast$. Note that $(b,g)$ is contained in $Z(T)$. Consider an arbitrary element $(a,f) \in J(Z(T)).$ Because of $a \in J(Z(A)),$ we have $ab = 0.$ Moreover, for any $x \in A$, we obtain $(ag+fb)(x) = g(xa) + f(bx) = 0$ since we have $bx = xb \in Ab \subseteq K(A) \subseteq \Ker(f)$ and $xa \in A \cdot J(Z(A)) \subseteq \Ker(g).$ This shows $(a,f) \cdot (b,g) = 0$ and hence $(b,g) \in \soc(Z(T)).$
\bigskip

Finally, using (i) and (v), we obtain $$R(T) = \soc(T) \cap Z(T) = 0 \oplus (A/K(A) + J(A))^\ast,$$ which proves~(vii). We remark that this statement, for fields of positive characteristic, is already proven in~\cite{BES07}. 
\end{proof}

\begin{Remark}
For $b \in A$, requiring $Ab \subseteq K(A)$ as in Lemma \ref{lemma:subspacest}\,(iv) forces $b \in J(A)$: For $A' \coloneqq A/J(A)$ and $b' \coloneqq b + J(A)$, we have $A' b' \subseteq K(A')$. Since $A'$ is semisimple and hence symmetric, we have $A' b' = 0$ since $K(A')$ does not contain any nontrivial left ideal. This implies $b \in J(A)$. 
\end{Remark}
 
\begin{theorem}\label{theo:soctaideal}
Let $A$ be an $F$-algebra with trivial extension algebra~$T \coloneqq T(A)$. 
\begin{enumerate}[(i)]
\item $J(Z(T))$ is an ideal in $T$ if and only if $J(Z(A))$ and $K(A)$ are ideals in~$A$. 
\item $\soc(Z(T))$ is an ideal in $T$ if and only if $I \coloneqq K(A) + A \cdot J(Z(A))$ and $S \coloneqq \{b \in \soc(Z(A)) \colon Ab \subseteq K(A)\}$ are ideals of $A$.  
\end{enumerate}
\end{theorem}

\begin{proof}
$\null$
\begin{enumerate}[(i)]
\item By Lemmas \ref{lemma:condsocleprod} and \ref{lemma:subspacest}, $J(Z(T))$ is an ideal of $T$ if and only if we have $$0 = J(Z(T)) \cdot K(T) = J(Z(A)) \cdot K(A) \oplus J(Z(A)) \cdot [A,A^\ast] + (A/K(A))^\ast \cdot K(A).$$ The condition $J(Z(A)) \cdot K(A) = 0$ is equivalent to $J(Z(A))$ being an ideal of $A$ (see Lemma \ref{lemma:condsocleprod}). If $J(Z(A))$ is an ideal of $A$, we have $$J(Z(A)) \cdot [A,A^\ast] = [J(Z(A)) \cdot A,A^\ast] =  [J(Z(A)),A^\ast] = 0.$$ Thus the second component of $J(Z(T)) \cdot K(T)$ is zero if and only if $(A/K(A))^\ast \cdot K(A) $ is. This is equivalent to $A \cdot K(A) = K(A) \cdot A$ being contained in $K(A)$ (see Lemma \ref{lemma:commutatorsmallestideal}), that is, to $K(A)$ being an ideal in $A$.

\item Again, $\soc(Z(T))$ is an ideal in $T$ if and only if we have $$0 = \soc(Z(T)) \cdot K(T) = S \cdot K(A) \oplus S \cdot [A,A^\ast] + (A/I)^\ast \cdot K(A).$$
If $S \cdot K(A) = 0$ holds, then we have $Ab \subseteq Z(A)$ for all $b \in S$ (see Lemma \ref{lemma:condsocleprod}\,(i)). Since $Ab$ annihilates $J(Z(A))$, we even obtain $Ab \subseteq \soc(Z(A))$. Moreover, we have $A(Ab) \subseteq Ab \subseteq K(A)$ and hence $Ab \subseteq S$. This shows that $S$ is an ideal of $A$. Conversely, if $S$ is an ideal of $A$, then $S \cdot K(A) = 0$ follows by \cite[Lemma 2.1]{KUL20}. If $S$ is an ideal of $A$, we have $$S \cdot [A, A^\ast] = [S, A^\ast] = 0.$$ Hence the second component of $\soc(Z(T)) \cdot K(T)$ is zero if and only if $(A/I)^\ast \cdot K(A) = 0$ holds. As before, this is equivalent to $A \cdot K(A) = K(A) \cdot A \subseteq I$, that is, to $I$ being an ideal of~$A$. \qedhere
\end{enumerate}
\end{proof}

\begin{Remark}
In the special case that the algebra $A$ itself is symmetric, $J(Z(T))$ is an ideal of $T$ if and only if $A$ is commutative (see Remark~\ref{rem:ka}). 
\end{Remark}

%
%
%

\section{Symmetric local algebras of small dimension}\label{sec:smalldimension}
Let $F$ be an algebraically closed field. Landrock showed in \cite{LAN20} that $J(Z(A))$ is an ideal in $A$ for every symmetric local $F$-algebra $A$ of dimension at most ten. In this section, we extend his result by proving that
$J(Z(A))$ is an ideal in every symmetric local $F$-algebra $A$ of dimension at most eleven. We will also show
that this bound is sharp. Moreover, we study the analogous problem for the socle of the center. In the following, we write $J^i$ for the powers $J^i(A) \coloneqq J(A)^i$ of the Jacobson radical of $A$.

\begin{lemma}[{\cite[Lemma E]{KUL84}}]\label{lemma:kultheoe}
	Let $I$ be an ideal of an $F$-algebra $A$ and let $n$ be a positive integer. Suppose that
	$$I^n = F\left\{x_{i1} \cdots  x_{in} \colon i = 1, \ldots, d\right\} + I^{n+1}$$
	holds for elements $x_{ij} \in I.$ 
	Then we have
	$$I^{n+1} = F\left\{x_{j1} x_{i1} \cdots x_{in} \colon i,j = 1, \ldots, d\right\} + I^{n+2}$$
	and 
	$$I^{n+1} = F\left\{x_{i1} \cdots x_{in} x_{jn} \colon i,j = 1, \ldots, d\right\} + I^{n+2}.$$
\end{lemma}

\begin{lemma}[{\cite[Lemma 2.7]{CHL91}}]\label{lemma:chlab}
Let $A$ be a local algebra with $\dim A/K(A) = 4$ and $\dim J^2/J^3 \geq 2$. Then either $\dim A \leq 8$ holds or there exist elements $a,b \in J$ such that $a^2 + J^3$ and $ab + J^3$ or $a^2 + J^3$ and $ba + J^3$ are linearly independent in $J^2/J^3$. 
\end{lemma}

\begin{lemma}[{\cite[Lemma 0.3]{CHL91} and \cite[Lemma G]{KUL84}}]\label{lemma:chlz}\label{lemma:onedimensionallayersymmetric}
Suppose that $A$ is a local algebra. If $\dim J^i /J^{i+1} = 1$
holds for some positive integer $i$, then we have $J^i \subseteq Z(A)$. If $A$ is additionally symmetric, then we even have $J^{i-1} \subseteq Z(A)$. 
\end{lemma}

The next statement can be found in the proof of \cite[Theorem 3.2]{LAN20}.

\begin{lemma}\label{lemma:centerdim3greater}
	Let $A$ be a non-commutative symmetric algebra. Then $\dim A \geq \dim Z(A) + 3$ holds.
\end{lemma}

\begin{proof}
	Since $A$ is not commutative, we have $\dim Z(A) < \dim A$. If $\dim A = \dim Z(A) +1$ holds, then we have $A = Fx \oplus Z(A)$ for some $x \in A$ and hence $A$ is commutative, a contradiction. If $\dim A = \dim Z(A) + 2$ holds, we may write $A = Fx \oplus Fy \oplus Z(A)$ with $x, y\in A$. This yields $K(A) \subseteq F[x,y]$ and hence $1 \geq \dim K(A) = \dim A - \dim Z(A) = 2$, which is a contradiction.
\end{proof}

We now collect some properties of symmetric local algebras. 

\begin{lemma}[{\cite[Lemma 3.1]{LAN17}}]\label{lemma:propertiessymmetriclocal}
Let $A$ be a symmetric local algebra. Then:
\begin{enumerate}[(i)]
	\item $\dim \soc(A) = 1$ and $\soc(A) \subseteq \soc(Z(A)).$
	\item $K(A) \cap \soc(A) = 0$ and $Z(A)$ is local.
	\item We have $J^{\ell-1} = \soc(A),$ where $\ell$ denotes the minimal positive integer with $J^\ell = 0$.
\end{enumerate}
\end{lemma}


\begin{theorem}\label{theo:kultheob} \label{theo:propertiessymmetricalgebra}
Let $A$ be a symmetric local algebra. If $\dim Z(A) \leq 4$ holds, then $A$ is commutative. For $\dim Z(A) = 5$, one of the following cases occurs: 
	\begin{enumerate}[(i)]
		\item $\dim A = 5$ and $A$ is commutative. 
		\item $\dim A = 8$ and there are two possibilities for the Loewy structure of $A$: 
		\begin{enumerate}[(a)]
			\item $\dim J/J^2 = \dim J^2/J^3 = 3$ and $\dim J^3= 1$, or 
			\item $\dim J/J^2 = \dim J^2/J^3 = \dim J^3/J^4 = 2$ and $\dim J^4= 1$.
		\end{enumerate} 
	\end{enumerate}
\end{theorem}

\begin{proof}
The result for $\dim Z(A) \leq 4$ is the statement of \cite[Theorem B]{KUL84}. Now let $\dim Z(A) = 5$. The fact that the algebra $A$ is of dimension five or eight is the main result of \cite{CHL92}. If $A$ is of dimension five, then $A$ is commutative, so let $\dim A = 8$. If $J^3 = 0$ holds, then Lemma~\ref{lemma:propertiessymmetriclocal} yields $\dim J^2/J^3 \leq 1$ and by Lemma~\ref{lemma:onedimensionallayersymmetric}, we obtain $J \subseteq Z(A)$. This implies that $A$ is commutative, which is a contradiction. Hence we have $J^3 \neq 0$, which yields $\dim J^3/J^4 \geq 1$ by Nakayama's lemma. If $\dim J^3/J^4 \geq 2$ holds, then Lemma~\ref{lemma:kultheoe} yields $\dim J/J^2 \geq 2$ and $\dim J^2/J^3 \geq 2$. Furthermore, we have $\dim J^4/J^5 \geq 1$ by Lemma \ref{lemma:propertiessymmetriclocal}. As $\dim A = 8$ and $\dim A/J = 1$ hold, this implies that $A$ has the Loewy structure given in (b). 
	It remains to consider the case $\dim J^3/J^4 = 1.$ Here, we obtain $J^2 \subseteq J(Z(A))$ by Lemma~\ref{lemma:onedimensionallayersymmetric} and hence $\dim J^2 \leq 4$ follows. On the other hand, we have $K(A) \subseteq J^2$ and hence $$(J^2)^\perp \subseteq K(A)^\perp \cap J = Z(A) \cap J = J(Z(A)),$$ which yields $\dim\, (J^2)^\perp \leq \dim J(Z(A)) = 4.$ Lemma \ref{lemma:propertiesperp}\,(i) then implies $\dim J^2 = \dim\, (J^2)^\perp = 4$ and hence $J^2 = J(Z(A)) = (J^2)^\perp.$ By Lemma \ref{lemma:propertiesperp}\,(v), this implies $J^4 = J^2 \cdot J^2 = J^2 \cdot (J^2)^\perp = 0$ and hence we obtain $\dim J^3 = 1$, so $A$ has the Loewy structure given in (a).
\end{proof}

We conclude this part with the following results on symmetric local $F$-algebras with a six-dimensional center, which are proven in \cite{DEI93}. 

\begin{lemma}[{\cite[Main theorem]{DEI93}}]\label{lemma:maintheoremdeiml}
Let $A$ be a symmetric local $F$-algebra with $\dim Z(A) = 6$ and $\dim J/J^2 = 2$. Then $\dim A \leq 12$ holds.
\end{lemma}

\begin{lemma}[{\cite[Lemma 2.6]{DEI93}}]\label{lemma:abcdeiml}
Let $A$ be a symmetric local $F$-algebra of dimension at least $11$ with $\dim Z(A) = 6$ and $\dim J/J^2 = \dim J^2/J^3 = 3$. By possibly replacing $A$ by its opposite algebra, we find elements $a,b,c \in J$ with $J = F\{a,b,c\} + J^2$ such that either $J^2 = F\{a^2, ab, ac\} + J^3$ or $J^2 = F\{a^2, ab,ba\} + J^3$ holds. 
\end{lemma}

\begin{Remark}
We take the opportunity to point out that Gerhard proved in her
diploma thesis \cite{GER93} that any symmetric local algebra $A$ satisfying $\dim Z(A) = 6$ and $\dim J^2/J^3 = 3$ has dimension at most~21. 
\end{Remark}

\subsection{Jacobson radical}

In this section, we investigate the Jacobson radical of $Z(A)$. We prove that $J(Z(A))$ is an ideal in $A$ if $A$ is a symmetric local algebra of dimension at most eleven and we provide an example of a symmetric local $F$-algebra $A$ of dimension twelve in which $J(Z(A))$ is not an ideal. This extends \cite[Theorem 3.2]{LAN20}, in which it is shown that $J(Z(A))$ is an ideal in $A$ if $A$ is a symmetric local algebra of dimension at most ten. 

\begin{theorem}\label{theo:dim11jzaideal}
	Let $A$ be a symmetric local $F$-algebra of dimension $\dim A \leq 11$. Then $J(Z(A))$ is an ideal in~$A$.
\end{theorem}

\begin{proof}
Let $A$ be a symmetric local algebra of minimal dimension in which $J(Z(A))$ is not an ideal, and assume $\dim A \leq 11$. By Lemma \ref{lemma:aicommutative}, there exists a nonzero ideal $I$ of $A$ such that $A' \coloneqq A/I$ is symmetric and non-commutative. By Theorem \ref{theo:kultheob} and Lemma \ref{lemma:centerdim3greater}, this implies $\dim Z(A') \geq 5$ and $\dim A' \geq 8$. By Lemma \ref{lemma:quotientalgebrasymmetric}, there exists an element $z \in J(Z(A))$ with $I = (Az)^\perp$. Note that we have $\dim Az = \dim A'$ and $Az \subseteq J$. If $\dim Az > 8$ holds, then there exists $x \in J$ with $J = Fx + Az = Fx + Fz + J^2$ since we have $\dim J \leq 10$. Since $A$ is then generated by $x$ and $z$, the algebra is commutative, a contradiction. Hence we have $\dim Az = 8$ and $\dim Z(A') = 5$. Moreover, $z' \coloneqq z+I$ is contained in $J(Z(A'))$ and we have $\dim A'/A'z' \leq \dim A/Az \leq 3$, which implies $\dim A'z' > 4$. Since $\dim A' < \dim A$ holds, $J(Z(A'))$ is an ideal in $A'$ by assumption. In particular, $A'z'$ is contained in $J(Z(A'))$ and hence we have $\dim A'z' \leq \dim J(Z(A')) = 4$, a contradiction.  
\end{proof}

\begin{Example}
	Let $F$ be an algebraically closed field of odd characteristic. We consider the unitary subalgebra $A$ of $\Mat_{12}(F)$ generated by the matrices 
	$$M = \begin{pmatrix}
	. & . & . & . & . & . & . & . & . & . & . & .\\
	1 & . & . & . & . & . & . & . & . & . & . & .\\
	. & . & . & . & . & . & . & . & . & . & . & .\\
	. & 1 & . & . & . & . & . & . & . & . & . & .\\
	. & . & 1 & . & . & . & . & . & . & . & . & .\\
	. & . & . & 1 & . & . & . & . & . & . & . & .\\
	. & . & . & . & 1 & . & . & . & . & . & . & .\\
	. & . & . & . & . & 1 & . & . & . & . & . & .\\
	. & . & . & . & . & . & 1 & . & . & . & . & .\\
	. & . & . & . & . & . & . & 1 & . & . & . & .\\
	. & . & . & . & . & . & . & . & 1 & . & . & .\\
	. & . & . & . & . & . & . & . & . & 1 & . & .\\
	\end{pmatrix}$$
	and
	$$N = \begin{pmatrix}
	. & . & . & . & . & . & . & . & . & . & . & .\\
	. & . & . & . & . & . & . & . & . & . & . & .\\
	1 & . & . & . & . & . & . & . & . & . & . & .\\
	. & . & 1 & . & . & . & . & . & . & . & . & .\\
	. & -1 &. & . & . & . & . & . & . & . & . & .\\
	. & . & . & . &-1 & . & . & . & . & . & . & .\\
	. & . & . & 1 & . & . & . & . & . & . & . & .\\
	. & . & . & . & . & . & 1 & . & . & . & . & .\\
	. & . & . & . & . &-1 & . & . & . & . & . & .\\
	. & . & . & . & . & . & . & . &-1 & . & . & .\\
	. & . & . & . & . & . & . & 1 & . & . & . & .\\
	. & . & . & . & . & . & . & . & . & . & 1 & .\\
	\end{pmatrix}.
	$$
	Here, zero entries are represented by dots. A short computation shows $M^7 = M^5 N = NM + MN = N^2-M^2 = 0$. Hence the set
	$$B \coloneqq \left\{\id,M, N, M^2, MN, M^3, M^2N, M^4, M^3 N, M^5, M^4N, M^6\right\}$$ generates $A$ as an $F$-vector space. One easily verifies that these elements are linearly independent, so $B$ is an $F$-basis of $A$. In particular, we obtain $\dim A = 12$. Since all nontrivial basis elements are nilpotent, the algebra $A$ is local.
	\bigskip
	
	Let $s \in \soc(A)$ and write $s = \sum_{v \in B} c_v v$ with coefficients $c_v \in F.$ The condition $s M^6 = 0$ translates to $c_{\id} = 0$ since all other products vanish. Due to $M^5 N = M^7 = 0$, we obtain $0 = s \cdot M^5 = c_M \cdot  M^6,$ so $c_M = 0$, and $0 = s \cdot M^4 N = c_N \cdot  M^6,$ which yields $c_N = 0.$ Continuing this way yields $s= c_{M^6} \cdot  M^6$, so $\soc(A) = F M^6$ is one-dimensional. 
	\bigskip
	
	By directly calculating the commutators of the elements in $B$, we see that an $F$-basis of $K(A)$ is given by $\{MN, M^2N, M^3N, M^4N, M^3, M^5\}$ and hence we obtain $\dim K(A) = 6$. Note that $K(A) \cap \soc(A) = 0$ holds. Thus we can define a linear form $\lambda \colon A \to F$ by setting $\lambda(M^6) = 1$ and $\lambda(b) = 0$ for $b \in B \backslash \{M^6\}$, and extending this $F$-linearly to $A$. Then the kernel of $\lambda$ does not contain any nonzero left ideals of $A$. Moreover, $\lambda$ vanishes on $K(A)$. In particular, $\lambda(ab) - \lambda(ba) = \lambda(ab-ba) = \lambda([a,b]) = 0$ holds and hence $\lambda$ is symmetric. By \cite[Theorem IV.2.2]{SKO11}, this shows that $(A,\lambda)$ is a symmetric local algebra. 
	\bigskip
	
	It is easily verified that the set $\{\id, M^2, M^4, M^5, M^4N, M^6\}$ is contained in $Z(A)$. This is even an $F$-basis of $Z(A)$ since we have $\dim Z(A) = \dim A - \dim K(A) = 6$. As $M^2$ is nilpotent, we have $M^2 \in J(Z(A))$. However, $M \cdot M^2 = M^3$ is not contained in $Z(A)$, so $J(Z(A))$ is not an ideal of $A$. 
\end{Example}

\subsection{Socle}
We now investigate the corresponding problem for the socle. First we show that $\soc(Z(A))$ is an ideal of $A$ if $A$ is a symmetric local algebra of dimension at most $16$. In the second part of this section, we prove that there exists a local trivial extension algebra $T$ of dimension $20$ with $\soc(Z(T)) \not \trianglelefteq T$. The dimensions 17, 18 and 19 remain open. 

\begin{theorem}\label{theo:dim16}
	Let $A$ be a symmetric local $F$-algebra. If $\dim A \leq 16$ holds, then $\soc(Z(A))$ is an ideal in $A$.
\end{theorem}

\begin{proof}
Assume that $A$ is a symmetric local algebra of dimension at most $16$ in which $\soc(Z(A))$ is not an ideal. By Lemma~\ref{lemma:condsocleprod}, there exists an element $z \in \soc(Z(A))$ with $z \cdot K(A) \neq 0$. By Lemma~\ref{lemma:quotientalgebrasymmetric}, $A' \coloneqq A/(Az)^\perp$ is a symmetric local algebra. Since $K(A)$ is not contained in $(Az)^\perp$, the algebra $A'$ is not commutative, which yields $\dim Az = \dim A' \geq \dim Z(A') + 3 \geq 8$	by Theorem \ref{theo:kultheob} and Lemma~\ref{lemma:centerdim3greater}. Since $z \in \soc(Z(A)) \subseteq J(Z(A))$ holds (see Lemma~\ref{lemma:socinj}), we obtain $z^2 = 0.$ In particular, we have $Az \subseteq (Az)^\perp$ and hence Lemma \ref{lemma:propertiesperp} yields $\dim A = \dim Az + \dim\, (Az)^\perp \geq 2 \cdot \dim Az \geq 16.$ By assumption, we then obtain $\dim A = 16$, which yields $\dim A' = \dim Az = \dim\, (Az)^\perp = 8$ and hence $Az = (Az)^\perp$. 
\bigskip

We also see $\dim Z(A') = 5$ and
conclude that $3 = \dim Z(A')^\perp = \dim K(A') = \dim K(A)+Az/Az$ holds. This implies $\dim K(A) + Az = 11$ and $5 = \dim (K(A) + Az)^\perp= \dim Z(A) \cap Az$. Moreover, $z \in \soc(Z(A))$ forces $$J(Z(A)) \subseteq Z(A) \cap (Az)^\perp  = Z(A) \cap Az \subseteq J(Z(A)),$$ so that $J(Z(A)) = Z(A) \cap Az$ and $\dim Z(A) = 6$. Hence we have $\dim K(A) = 10$ and $\dim K(A) \cap Az = 10 + 8 - 11 = 7$. Since $\soc(A) \subseteq Az$ and $\soc(A) \cap K(A) = 0$ we obtain $Az = (K(A) \cap Az) \oplus \soc(A) \subseteq J^2$, and this yields $\dim J/J^2 = \dim J(A')/J(A')^2$. 
\bigskip

The two possible Loewy structures of $A'$ are given in Theorem \ref{theo:kultheob}. If $A'$ has a Loewy structure of type (b), we have $ \dim J/J^2 = \dim J(A')/J^2(A') = 2$. By Lemma \ref{lemma:maintheoremdeiml}, this yields the contradiction $\dim A \leq 12$. 
Hence the Loewy structure of $A'$ is of type (a) in Theorem \ref{theo:kultheob}. We have $\dim J/J^2 = \dim J(A')/J^2(A') = 3$
and $\dim J^2/J^3 \geq \dim J^2(A')/J^3(A') = 3$.
\bigskip

Writing $J = F\{a,b,c\} + J^2$ for elements $a,b,c \in J$, we obtain $$K(A) \subseteq F\{[a,b], [b,c], [a,c]\} + J^3.$$ Thus $\dim A/J^3 \leq 3 + \dim A/K(A) + J^3$ follows. On the other hand, we have
$Az = (K(A) \cap Az) + \soc(A) \subseteq K(A) + J^3$, which implies
$\dim A/K(A) + J^3 = \dim A'/K(A') + J(A')^3 < \dim A'/K(A') = 5$.
Hence $\dim A/J^3 \leq 7 = \dim A'/J(A')^3 \leq \dim A/J^3$, so
that $\dim A/J^3 = 7$ and $\dim J^3 = 9$. We also conclude that
$Az \subseteq J^3$ and $Jz \subseteq J^4$ hold.

%
%
\bigskip

On the other hand, we have $J(A')^4 = 0$ and hence $J^4 \subseteq Az$. Since $J^3$ is not contained in $Z(A)$, Lemma~\ref{lemma:onedimensionallayersymmetric} yields $\dim J^3/J^4 \geq 2$ and hence $J^4$ is a proper subset of $Az$. As in Lemma~\ref{lemma:propnustar}, the map 
$$\varphi \colon A' \to Az,\ a+ Az \mapsto az$$ is an $A$-bimodule isomorphism. With this, we see that $Jz$ is the unique maximal submodule of $Az$ and hence $J^4 = Jz$ follows. Furthermore, we obtain $\dim J^5  = \dim J^2z = 4$ and $\dim J^6 = \dim J^3 z = 1$. 
\bigskip

If we have $J^3 = x J^2 + J^4$ for some $x \in J$, then $J^4 = x^2 J^2 + J^5$ follows and we obtain the contradiction $\dim J^4/J^5 \leq \dim J^3 /J^4 = 2$. Hence $\dim x J^2 + J^4/J^4 < \dim J^3/J^4 = 2$ holds for all $x \in J$. Similarly, we obtain $\dim J^2 x + J^4/J^4 < 2$ for all $x \in J$. 
By Lemma \ref{lemma:abcdeiml}, there exist elements $a,b,c \in J$ with $J = F\{a,b,c\} + J^2$ and $J^2 = F\{a^2, ab,ba \} + J^3$. By Lemma \ref{lemma:kultheoe}, we obtain $J^3 = F\{a^3, a^2b, aba, ba^2, bab,b^2a\} + J^4$. Assume that $a^3$ is not contained in $J^4$. Then $a J^2$ and $J^2a $ are contained in $Fa^3 + J^4$. Hence $J^3 = F\{a^3, bab\} + J^4$ follows, so $bab$ is not contained in $J^4$. This yields the contradiction $$J^4 = F\{a^4, abab, ba^3, b^2 ab\} + J^5 = F\{a^4, a^3b\} + J^5.$$
Hence $a^3$ is contained in $J^4$. Now assume $a^2b \notin J^4$. Then $J^3 = F\{a^2b, ba^2, b^2a\} + J^4$ holds. First let $J^3 = F\{a^2b, ba^2\} + J^4$. Then $ba^2$ is not contained in $J^4$ and we obtain the contradiction
$$J^4 = F\{a^3b,aba^2,ba^2b,b^2a^2\} + J^5 = F\{ba^2b,b^2a^2\} + J^5,$$ since
$a^3b \in J^5$ and $aba^2 \in Fba^3 + J^5 = J^5$ hold. Analogously, let $J^3 = F\{a^2b,b^2a\} + J^4$, that is, $b^2a$ is not contained in $J^4$. This yields the contradiction
$$J^4 = F\{a^3b,ab^2a,ba^2b,b^3a\} + J^5 = F\{ba^2b,b^3a\} + J^5$$ due to $a^3b \in J^5$ und $ab^2a \in Fa^3b + J^5 = J^5$. Hence we may assume $a^2b \in J^4$. In case that $aba$ is not contained in $J^4$, we obtain $J^3 = F\{aba,bab\}$ and $bab$ is not contained in $J^4$. This yields the contradiction
$$J^4 = F\{a^2ba,abab,baba,b^2ab\} + J^5 = F\{abab,baba\} + J^5$$ as in the previous cases. Hence we may additionally assume $aba \in J^4$. This implies $J^3 =  bJ^2 + J^4$, which is a contradiction. 
\end{proof}

For local trivial extension algebras, we obtain the following result:

\begin{lemma}
Let $A$ be a local algebra of dimension $\dim A \leq 9$. Then: 
\begin{enumerate}[(i)]
\item $I \coloneqq K(A) + A \cdot J(Z(A))$ is an ideal of $A$. 
\item $S \coloneqq \{b \in \soc(Z(A)) \colon Ab \subseteq K(A)\}$ is an ideal of $A$. 
\item For the corresponding trivial extension algebra $T \coloneqq T(A)$, we have $\soc(Z(T)) \trianglelefteq T$. 
\end{enumerate}
\end{lemma}

\begin{proof}
$\null$
\begin{enumerate}[(i)]
\item Assume that $A$ is a counterexample. Then $A \cdot K(A) = K(A) \cdot A$ is not contained in $I$. Since $A = F \cdot 1 \oplus J$ holds, we obtain $JK(A) \not \subseteq I$. Because of $K(A) = [A,A] = [J,J] \subseteq J^2$, we have $JK(A) \subseteq J^3$, so $J^3$ is not contained in $I$ and, in particular, not in $J(Z(A))$. Then $J/J^2$, $J^2/J^3$ and $J^3/J^4$ are of dimension at least two (see Lemma \ref{lemma:chlz}), which implies $\dim J^4 \leq 2$. This yields $J^4 \subseteq Z(A)$ since for $J^5 = 0$, we obtain $[A,J^4] = [J,J^4] \subseteq J^5 = 0$ and for $\dim J^5 = 1$, we have $\dim J^4/J^5 \leq 1$ and the statement follows by Lemma \ref{lemma:chlz}.
On the other hand, $J^3 \not \subseteq J(Z(A))$ implies $J^4 \neq 0$, so $\dim J/J^2 \leq 3$ holds. Furthermore, we have $\dim K(A) + J^3/J^3 \leq 2$ since otherwise $\dim J^2/J^3 \geq 3$ and hence $\dim J/J^2 \leq 2$ follows. Writing $J = F\{x,y\} + J^2$ for some $x,y  \in J$ then yields the contradiction $K(A) = [J,J] \subseteq F[x,y] + J^3$. 
\bigskip

Set $B \coloneqq A/J^4$ and $N \coloneqq J(B)$. Then $N^4 = 0$ holds and there exist $a,b,c \in N$ with $N = F\{a,b,c\} + N^2$. Furthermore, $\dim K(B) + N^3/N^3 \leq 2$ holds and $NK(B)$ is not contained in $K(B) \subseteq F\{[a,b],[a,c],[b,c]\} + N^3 \subseteq N^2$. Hence we have $N K(B) = N[a,b] + N[a,c] + N[b,c]$. Due to symmetry reasons, we may assume that $N[a,b]$ is not contained in $K(B)$. Since $K(B)$ contains the elements $a[a,b]$ and $b[a,b]$, it follows that $c[a,b]$ is not contained in $K(B)$. On the other hand, since $\dim K(B) + N^3/N^3 \leq 2$ holds, there exists a nonzero tuple $(\alpha, \beta, \gamma) \in F^3$ with
$$\alpha [a,b] + \beta [a,c] + \gamma [b,c] \in N^3.$$
If $\alpha$ is nonzero, then $[a,b]$ is contained in
$F\{[a,c],[b,c]\} + N^3$. We then obtain the contradiction $c[a,b] \in F\{c[a,c], c[b,c]\} \subseteq K(B)$.
Hence we have $\alpha = 0$ and then $[\beta a + \gamma b, c] \in N^3$. Due to symmetry reasons, we may assume $\gamma \neq 0$. By dividing by $\gamma$, we may even assume $\gamma =1$.
But then we have $[\beta a + b,c] a \in
N^4 = 0$ and we obtain the contradiction
$$c[a,b] = c[a,\beta a + b] = ca(\beta a + b) - c(\beta a + b)a
= ca(\beta a + b) - (\beta a + b)ca = [ca, \beta a + b]
\in K(B).$$ 
\item Assume that there exists an element $z \in \soc(Z(A))$ with $Az \subseteq K(A)$ such that $Az \not \subseteq Z(A)$ holds. Since we have $z \in K(A) \subseteq J^2$, this yields $J^3 \not \subseteq Z(A)$ and hence $J^4 \neq 0$. As in the proof of (i), it follows that $J/J^2$, $J^2 /J^3$ and $J^3 /J^4$ are at least two-dimensional and we have $\dim J/J^2 \leq 3$ as well as $J^4 \subseteq Z(A)$. In particular, we obtain $z \notin J^3$, but $z^2 = 0$ holds due to $z \in \soc(Z(A))$.
\bigskip

By Lemma \ref{lemma:condsocleprod}, we have $K(A) \cdot z \neq 0$. Hence there exist elements $a,b \in J$ with $0 \neq [a,b] z = [az,b] = [a,bz]$. This yields $az \notin Z(A)$, so $az \notin J^4$ holds. Due to $[Fa,Faz + J^4] = [Fa,J^4]
= 0$, we have $bz \notin Faz + J^4$, so $az + J^4$ and $bz + J^4$ are linearly independent in $J^3/J^4$. Hence $a + J^2$ and $b + J^2$ are linearly independent in $J/J^2$.
\bigskip

Assume that $[a,b]+J^3$ and $z+J^3$ are linearly dependent in $J^2/J^3$. Then we have $[a,b] = \beta z + y$ for some $\beta \in F$ and $y \in J^3$. Hence $0 \neq [a,b]z = \beta z^2 + yz = yz$ follows. In particular, we have $J^5 \neq 0$. Hence the Loewy layers $J^i/J^{i+1}$
($i=0,...,5$) of $A$ are of dimensions 1,2,2,2,1,1. This implies $J^3 = F\{az,bz\} + J^4$ and $J^3 z = 0$, a contradiction to $yz \neq 0$. 
\bigskip

Hence $[a,b]+J^3$ and $z+J^3$ are linearly independent in $K(A)+J^3/J^3 \subseteq J^2/J^3$. In particular, we obtain $\dim K(A)+J^3/J^3 \geq 2$. As in (i), this implies $\dim J/J^2 = 3$. 
It follows that the Loewy layers $J^i/J^{i+1}$ ($i=0,...,4$) of $A$ are of dimensions 1,3,2,2,1. This yields $$J^2 = K(A) + J^3 =
K(A) + Az = K(A).$$ By Lemma \ref{lemma:chlab}, there exist elements $x,y, w\in J$ such that $J = F\{x,y,w\} + J^2$ holds and $x^2+J^3$ and $xy+J^3$ form a basis of $J^2/J^3$
(by possibly going over to the opposite algebra of $A$).
%
%
%
By Lemma~\ref{lemma:kultheoe}, $x^3 + J^4$ and $x^2y + J^4$ form an $F$-basis of $J^3/J^4$. Write $xw \equiv \alpha x^2 + \beta xy \pmod{J^3}$ for some $\alpha, \beta \in F$. By replacing $w$ by $\bar{w} \coloneqq w - \alpha x - \beta y$, we may assume $xw \in J^3$. Furthermore, there exist coefficients $\alpha_i, \beta_i \in F$ ($i = 1,\ldots, 4$) with 
	\begin{alignat*}{1}
	yw &\equiv \alpha_1 x^2 + \beta_1 xy \pmod{J^3} \\
	wx &\equiv \alpha_2 x^2 + \beta_2 xy \pmod{J^3} \\
	wy &\equiv \alpha_3 x^2 + \beta_3 xy \pmod{J^3} \\
	w^2 &\equiv \alpha_4 x^2 + \beta_4 xy \pmod{J^3}.
	\end{alignat*}
	With this, we obtain 
	\begin{alignat*}{4}
	0 &\equiv (xw) x &&\equiv x (wx) &&\equiv \alpha_2 x^3 + \beta_2 x^2y &&\pmod{J^4} \\
	0 &\equiv (xw) y &&\equiv x (wy) &&\equiv \alpha_3 x^3 + \beta_3 x^2y &&\pmod{J^4} \\
	0 &\equiv (xw) w &&\equiv x w^2 &&\equiv \alpha_4 x^3 + \beta_4 x^2y &&\pmod{J^4}.
	\end{alignat*}
	Comparing the coefficients yields $\alpha_2 = \beta_2 = \alpha_3 = \beta_3 = \alpha_4 = \beta_4 = 0$. This implies 
	$$0 \equiv yw^2 \equiv (yw) w \equiv \alpha_1 x^2w + \beta_1 xyw \equiv \beta_1 xyw \equiv \beta_1 (\alpha_1 x^3 + \beta_1 x^2y) \pmod{J^4},$$
	which yields $\beta_1 = 0$. Furthermore, we obtain 
	$$0 \equiv y (wy) \equiv (yw) y \equiv \alpha_1 x^2 y \pmod{J^4}$$ 
	and hence $\alpha_1 = 0$. This yields $[x,w], [y,w] \in J^3$ and hence $$J^2 = K(A) \subseteq F\{[x,y], [x,w], [y,w]\} + J^3 \subseteq F[x,y] + J^3,$$ which is a contradiction to $\dim J^2/J^3 = 2$.
	\item This follows by Theorem \ref{theo:soctaideal} together with (i) and (ii).  \qedhere
\end{enumerate}
\end{proof}

Hence every local trivial extension algebra $T$ of dimension at most $18$ satisfies $\soc(Z(T)) \trianglelefteq T$. The following example demonstrates that this bound is tight:

\begin{Example}\label{ex:soc20}
	Let $F$ be an algebraically closed field of characteristic two. We consider the unitary subalgebra $A$ of $\Mat_{10}(F)$ generated by the matrices 
	
	$$M = \begin{pmatrix}
	. & . & . & . & . & . & . & . & . & . \\
	1 & . & . & . & . & . & . & . & . & . \\
	. & . & . & . & . & . & . & . & . & . \\
	. & 1 & . & . & . & . & . & . & . & . \\
	. & . & 1 & . & . & . & . & . & . & . \\
	. & . & . & 1 & . & . & . & . & . & . \\
	. & . & . & . & 1 & . & . & . & . & . \\
	. & . & . & . & . & 1 & . & . & . & . \\
	. & . & . & . & . & . & 1 & . & . & . \\
	. & . & . & . & . & . & . & 1 & 1 & . \\
	\end{pmatrix} \text{ and  }\ 
	N = \begin{pmatrix}
	. & . & . & . & . & . & . & . & . & . \\
	. & . & . & . & . & . & . & . & . & . \\
	1 & . & . & . & . & . & . & . & . & . \\
	. & 1 & . & . & . & . & . & . & . & . \\
	. & 1 & . & . & . & . & . & . & . & . \\
	. & 1 & . & . & . & . & . & . & . & . \\
	. & 1 & . & 1 & 1 & . & . & . & . & . \\
	. & . & . & 1 & . & 1 & . & . & . & . \\
	. & . & . & . & 1 & 1 & . & . & . & . \\
	. & . & . & . & . & 1 & 1 & 1 & 1 & . \\
	\end{pmatrix}.
	$$
	Again, zero entries are represented by dots. One can check that these matrices satisfy the relations $M^6 = N^2 = 0$ as well as $NM = M^2 + MN + M^3 + M^2 N$. Moreover, $M^4N = M^5$ holds. It is easy to see that the set $B \coloneqq \{M^{\ell_1} N^{\ell_2} \colon \ell_1 \in \{0, \ldots, 4\},\ \ell_2 \in \{0,1\}\}$ is an $F$-basis for $A$. As before, we set $J \coloneqq J(A)$. 
	The Loewy layers $J^i/J^{i+1}$ ($i = 0, \ldots, 5$) of $A$ have the dimensions $1,2,2,2,2,1$, respectively. In particular, $A$ is a local algebra. One computes that $$J(Z(A)) = F\{M^4, M^5\} \subseteq F\left\{M^2, M^3 + M^3N, M^2  N + M^3 N, M^4, M^5\right\} = K(A)$$ holds. Thus $J(Z(A))$ is an ideal of $A$. Since $K(A)$ contains $M^2$, but not $M^3$, we obtain $K(A) + A \cdot J(Z(A)) = K(A) \neq A \cdot K(A)$. In particular, $K(A)$ is not an ideal of $A$. The trivial extension algebra $T \coloneqq T(A)$ has dimension $20$ and by Theorem \ref{theo:soctaideal}, $\soc(Z(T))$ is not an ideal of $T$.
\end{Example}

Combined with Theorem \ref{theo:dim16}, this yields $\dim A \in \{17,18,19,20\}$ for a symmetric local $F$-algebra $A$ of minimal dimension in which $\soc(Z(A))$ is not an ideal.

\section*{Acknowledgment}
The authors are grateful to Markus Linckelmann for a useful comment
concerning Section \ref{sec:symmetricalgebrasintro}. The results of this paper are part of the PhD-thesis of the first author, which she currently writes under the supervision of the second author. \nocite{BRE22}

\bibliographystyle{plain}
\bibliography{Symmetricalgebras.bib}

\clearpage

\appendix

\end{document}